\documentclass[12pt]{amsart}
\usepackage{amsmath,amsfonts,amssymb,amscd}
\usepackage{latexsym,epsfig}
\usepackage[latin1]{inputenc}
\usepackage[T1]{fontenc}
\usepackage{enumerate}
\usepackage{amsbsy}
\usepackage[matrix,arrow]{xy}

\newcommand{\N}{\ensuremath{\mathbb N}}
\newcommand{\Z}{\ensuremath{\mathbb Z}}
\newtheorem{thm}{Theorem}

\newtheorem{prop}[thm]{Proposition}

\newtheorem{rem}[thm]{Remark}

\newcommand{\eop}{%
  \relax
  \ifvmode[i)]
    \noindent
  \else
    \unskip
    \hskip0pt plus-1fill\relax
  \fi
  \vrule width0pt
  \nobreak
  \hfill
  {\hspace*{\fill}\vrule width3pt height8pt depth0pt}%
}

\begin{document}

\title{Groups of virtual and welded links}

\author[Bardakov]{Valeriy G. Bardakov}
\address{Sobolev Institute of Mathematics, Novosibirsk 630090, Russia}
\email{bardakov@math.nsc.ru}

\author[Bellingeri]{Paolo Bellingeri}
\address{Laboratoire de Math\'ematiques Nicolas Oresme, CNRS UMR 6139, Universit\'e de Caen BP 5186,  F-14032 Caen, France.
}
\email{paolo.bellingeri@unicaen.fr}


\subjclass{Primary 20F36}

\keywords{Braid groups, virtual braid, welded braids, virtual knots, welded knots, group of knot}


\begin{abstract}
We define new notions of  groups of virtual and
 welded knots (or links)
 and we study their relations with other invariants, in particular the
Kauffman group of a virtual knot. 
\end{abstract}

\maketitle

\section{introduction}

Virtual knot theory has been  introduced by Kauffman \cite{Ka}
as a generalization of classical knot theory.  Virtual knots (and links) are  represented as generic immersions of circles in the plane (virtual link diagrams)
where double points can be classical (with the usual information on overpasses and underpasses) or virtual.
Virtual link diagram are equivalent under ambient isotopy and some  types of local moves (generalized Reidemeister moves): classical
Reidemeister moves (Figure \ref{reidcl}), virtual Reidemeister  moves and mixed Reidemeister moves (Figures \ref{reidvirt} and
 \ref{reidmix}).

\begin{figure}[h]
 \centering
 \includegraphics[width=12cm,]{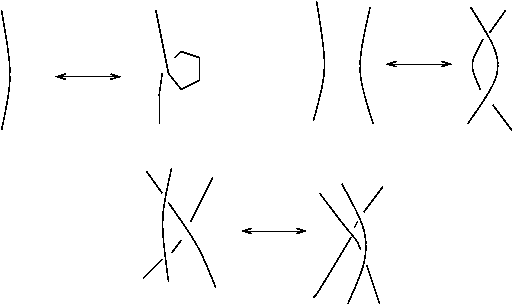}
 \caption{Classical Reidemeister moves} \label{reidcl}
\end{figure}

 \begin{figure}[h]
 \centering
 \includegraphics[width=11cm,]{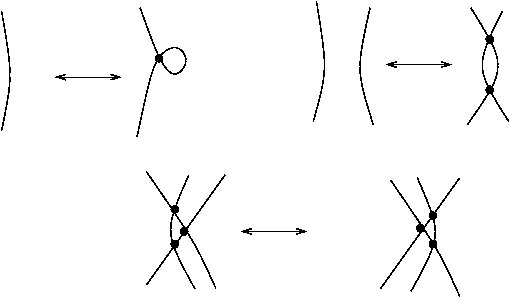}
  \caption{Virtual Reidemeister moves } \label{reidvirt}
\end{figure}

\begin{figure}[h]
\centering
 \includegraphics[width=7cm,]{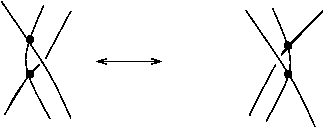}
\caption{Mixed Reidemeister moves} \label{reidmix}
\end{figure}

A Theorem of Goussarov, Polyak and Viro \cite[Theorem 1B]{GPV} states that if two classical knot diagrams are
equivalent under generalized Reidemeister moves, then they are equivalent under the classical
Reidemeister moves. In this sense virtual link theory is a nontrivial extension of  classical
theory. This Theorem is a straightforward consequence  of the fact that the knot group (more precisely the \emph{group system} of a knot, see for instance
\cite{CMG,J})
is a complete knot invariant which can be naturally extended in the realm of virtual links. Nevertheless this notion of invariant does not appear satisfactory for virtual objects (see Section~\ref{gvl}): the main goal of this paper is to explore new invariants for virtual links using braids and their virtual generalizations.

In fact, using virtual generalized Reidemeister moves we can introduce  a notion of ``virtual'' braids
(see for instance \cite{Ka,Ver01}). Virtual braids on $n$ strands form a group, usually denoted by $VB_n$.
The relations between virtual braids and virtual knots (and links) are completely
determined by a generalization
of Alexander and  Markov Theorems~\cite{K}.

To the generalized Reidemeister moves on  virtual diagrams one could add the following local  moves, called {\it forbidden moves} of type $F1$ and $F2$  (Figure \ref{reidforb}):

\begin{figure}[h]
\centering
 \includegraphics[width=12cm,]{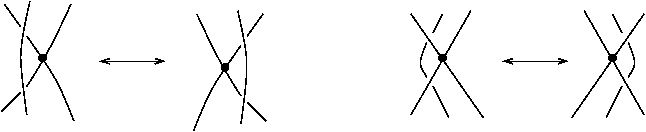}
\caption{Forbidden moves of type $F1$ (on the left) and type $F2$ (on the right)} \label{reidforb}
\end{figure}

 We can include one or both of them to obtain a "quotient" theory of the theory of
virtual links. If we allow  the move $F1$, then we obtain the theory of {\it Welded links} whose interest is growing up recently, in particular because of the fact that the welded braid counterpart can be defined in several equivalent ways (for instance in terms of configuration spaces, mapping classes and  automorphisms of free groups). The theory with
both forbidden moves added is called the theory of {\it Fused links}   but this theory is trivial, at least at the level of knots, since any knot  is equivalent to the trivial knot  \cite{Kan, Nel}.

The paper is organized as follows:
in Sections~\ref{intro} and~\ref{gl}   we recall some definitions and classical
results and  we construct a representation of $VB_n$ into $Aut \, F_{n+1}$;
using this representation we define (Section~\ref{gvl}) a new notion of group of a virtual knot (or link) and we compare our invariant to other
known invariants.
In the case of welded objects (Section~\ref{gwl})  our construction gives an invariant which is a straightforward generalization to welded knots
of Kauffman notion of group of virtual knots. We conclude with some observations on the analogous of Wada groups in the realm of welded links.

\medskip

{\bf Acknowledgements.}
The research of the first author was partially supported by  RFBR-10-01-00642.
The research of the second  author was partially supported  by  French grant ANR-11-JS01-002-01.

This  work started during the  staying of the
second author at the University of Caen, december 2010, in the framework  of the French - Russian grant 10-01-91056.
The first author would like to thank the members of the Laboratory  of
Mathematics of the University of Caen for their kind hospitality.


\section{Braids vs  virtual and welded braids} \label{intro}

 During last twenty years several  generalizations of  braid groups  were
defined and studied, according to their definition as "diagrams" in the plane:
in particular singular braids \cite{Ba},  virtual braids \cite{Ka, Ver01} and welded
braids \cite{FRR}.

It is worth to mention that  for all of these generalizations it exists an Alexander-like theorem, stating that
any singular (respectively virtual or welded) link  can be represented as the closure of a singular (respectively virtual or welded) braid. Moreover, there are generalizations of
classical Markov's theorem for braids  giving a characterization for two singular (respectively virtual or welded) braids whose closures represent the same singular (respectively virtual or welded) link \cite{G,K}.

In the following we introduce virtual and welded braid groups as quotients of free product of braid groups and corresponding symmetric groups.

Virtual and welded  braid groups have  several other definitions, more intrinsic, see for instance \cite{Bn,K}
 for the virtual case and \cite{BH,FRR,K} for the welded one.

The braid group $B_n$, $n\geq 2$, on $n$ strings can be defined as
the  group generated by $\sigma_1,\sigma_2,\ldots,\sigma_{n-1}$,
with the defining relations:
\begin{equation}
\sigma_i \, \sigma_{i+1} \, \sigma_i = \sigma_{i+1} \, \sigma_i \, \sigma_{i+1},~~~
i=1,2,\ldots,n-2, 
\nonumber
\end{equation}
\begin{equation}
\sigma_i \, \sigma_j = \sigma_j \, \sigma_i,~~~|i-j|\geq 2. 
\nonumber
\end{equation}

The virtual braid group $VB_n$ can be defined as the group generated by the elements $\sigma_i$, $\rho_i$, $i = 1, 2, \ldots, n-1$ with the defining relations:
\begin{equation}
\sigma_i \, \sigma_{i+1} \, \sigma_i = \sigma_{i+1} \, \sigma_i \, \sigma_{i+1},~~~
i=1,2,\ldots,n-2, 
\nonumber
\end{equation}
\begin{equation}
\sigma_i \, \sigma_j = \sigma_j \, \sigma_i,~~~|i-j|\geq 2. 
\nonumber
\end{equation}
\begin{equation}
\rho_i \, \rho_{i+1} \, \rho_i = \rho_{i+1} \, \rho_i \, \rho_{i+1},~~~
i=1,2,\ldots,n-2, 
\nonumber
\end{equation}
\begin{equation}
\rho_i \, \rho_j = \rho_j \, \rho_i,~~~|i-j|\geq 2. 
\nonumber
\end{equation}
\begin{equation}
\rho_i^2 =1,~~~~
i=1,2,\ldots,n-1; 
\nonumber
\end{equation}
\begin{equation}
\sigma_{i} \, \rho_{j} = \rho_{j} \, \sigma_{i},~~~|i-j| \geq 2, \qquad
\label{eq20}
\nonumber
\end{equation}
\begin{equation}
\rho_{i} \, \rho_{i+1}  \, \sigma_{i} = \sigma_{i+1} \, \rho_{i}  \, \rho_{i+1},~~~
i=1,2,\ldots,n-2.
\label{eq21}
\nonumber
\end{equation}

It is easy to verify that $\rho_i$'s generate
the symmetric group
$S_n$ and that
 the $\sigma_i$'s generate the braid group $B_n$ (see Remark \ref{weldedvirtualembedding}).

In  \cite{GPV} it was proved that the  relations
$$
\rho_{i} \, \sigma_{i+1}  \, \sigma_{i} = \sigma_{i+1} \, \sigma_{i}
\, \rho_{i+1},~~~~~~~~~~~~
\rho_{i+1} \, \sigma_{i}  \, \sigma_{i+1} = \sigma_{i}
\, \sigma_{i+1} \, \rho_{i}
$$
corresponding to the forbidden moves F1 and F2 for virtual link diagram, are not fulfilled in $VB_n$.

According to \cite{FRR} the welded braid group $WB_n$  is generated by $\sigma_i$, $\alpha_i$, $i=1, 2, \ldots, n-1$.
Elements
$\sigma_i$ generate the braid group $B_n$ and elements
$\alpha_i$  generate the symmetric group $S_n$, and the following mixed relations hold
\begin{equation}
\alpha_{i}  \, \sigma_{j} = \sigma_{j} \, \alpha_{i}, ~~~|i - j| \geq 2, \label{eq22}
\nonumber
\end{equation}
\begin{equation}
\alpha_{i+1} \, \alpha_{i} \, \sigma_{i+1} = \sigma_{i} \, \alpha_{i+1} \, \alpha_i,
~~~ i=1,2,\ldots,n-2, \label{eq23}
\nonumber
 \end{equation}
 \begin{equation}
\alpha_{i} \, \sigma_{i+1} \, \sigma_{i} = \sigma_{i+1} \, \sigma_{i} \, \alpha_{i+1},~~~ i=1,2,\ldots,n-2.
 \label{eq24}
\nonumber
\end{equation}

Comparing the defining relations of $VB_n$ and $WB_n$,  we see that
the group presentation of $WB_n$ can be obtained from the group presentation of $VB_n$
replacing  $\rho_i$ by $\alpha_i$ and adding  relations of type $
\alpha_{i} \, \sigma_{i+1} \, \sigma_{i} = \sigma_{i+1} \, \sigma_{i} \, \alpha_{i+1},~~~ i=1,2,\ldots,n-2
$ which are related to $F1$ moves.

Notice that  if we add to relations of $VB_n$ the relations related to  $F2$ moves:
 \begin{equation}
\rho_{i+1} \, \sigma_{i} \, \sigma_{i+1}=\sigma_{i} \, \sigma_{i+1} \, \rho_i  ,
~~~ i=1,2,\ldots,n-2, \label{eq25}
\nonumber
 \end{equation}
 we get a group, $WB'_n$, which is isomorphic to $WB_n$: this isomorphism is given by the map
$\iota_n : WB'_n \to WB_n$ that
sends   $\rho_i$ in  $\alpha_i$ and $\sigma_i$ in $\sigma_i^{-1}$.

\vspace{0.5cm}


\section{Braids as automorphisms of free groups and generalizations}\label{gl}

As remarked by Artin, the braid group $B_n$ may be represented as a subgroup of ${\rm Aut}(F_n)$
by associating to
any  generator $\sigma_i$, for $i=1,2,\ldots,n-1$, of $B_n$ the following
automorphism of $F_n$:
$$
\sigma_{i} : \left\{
\begin{array}{ll}
x_{i} \longmapsto x_{i} \, x_{i+1} \, x_i^{-1}, &  \\ x_{i+1} \longmapsto
x_{i}, & \\ x_{l} \longmapsto x_{l}, &  l\neq i,i+1.
\end{array} \right.
$$

Artin proved a stronger result (see for instance~\cite[Theorem 5.1]{LH}), by giving a characterization of
braids as automorphisms of free groups. He proved that
 any automorphism
$\beta $ of ${\rm Aut}(F_n)$ \footnote{In the following we will consider the action of (classical, virtual or  welded) braids from left to  right  and $\beta_1 \beta_2 (x_i)$ will denote  $((x_i) \beta_1) \beta_2.$}  corresponds to an element of  $B_n$
if and only if $\beta $ satisfies  the following conditions:
\begin{eqnarray*}
& i) & \beta(x_i) = a_i^{-1} \, x_{\pi(i)} \, a_i,~~1\leq i\leq n,\\
& ii)& \beta(x_1x_2 \ldots x_n)=x_1x_2 \ldots x_n,
\end{eqnarray*}
where $\pi \in S_n$
and $a_i \in F_n$.

The group of conjugating automorphisms
$C_n$ consists  of automorphisms satisfying the first condition.
In \cite{FRR} it was proved that $WB_n$ is isomorphic to  $C_n$ and therefore the group  $WB_n$  can be also considered as a subgroup of  $\mbox{Aut}(F_n)$. More precisely the generators $\sigma_1,\ldots, \sigma_{n-1}$ of $WB_n$
correspond to  previous automorphisms of $F_n$ while any generator $\alpha_i$, for $i=1,2,\ldots,n-1$
is associated  to the following automorphism of $F_n$:
$$
\alpha_{i} : \left\{
\begin{array}{ll}
x_{i} \longmapsto   x_{i+1}  &  \\ x_{i+1} \longmapsto
x_{i}, & \\ x_{l} \longmapsto x_{l}, &  l\neq i,i+1.
\end{array} \right.
$$

\begin{rem}\label{weldedvirtualembedding}
As noticed by Kamada \cite{Ka}, the above representation of  $WB_n$ as conjugating automorphisms
and the fact   $WB_n$  is a quotient of $VB_n$ imply that  the $\sigma_i$'s generate the braid group $B_n$ in $VB_n$.
\end{rem}

On the other hand the construction of an embedding  of $VB_n$  into
$\mbox{Aut}(F_m)$ for some $m$ remains an open problem.

\medskip

\begin{thm}\cite{Bar-1}  \label{theorem3}
 There is a representation $\psi$ of $VB_n$ in $\mbox{Aut}(F_{n+1})$,
$F_{n+1} = \langle x_1, x_2, \ldots, x_n, y \rangle$  which is defined by the following actions on
the generators of $VB_n$:
$$
\psi(\sigma_{i}) : \left\{
\begin{array}{ll}
x_{i} \longmapsto x_{i} \, x_{i+1} \, x_i^{-1}, &  \\
x_{i+1} \longmapsto x_{i}, &\\
x_{l} \longmapsto x_{l}, ~~~  l\neq i,i+1; & \\
y \longmapsto y,
\end{array} \right.
~~~
\psi(\rho_{i}) : \left\{
\begin{array}{ll}
x_{i} \longmapsto y \, x_{i+1} \, y^{-1}, &  \\
x_{i+1} \longmapsto y^{-1} \, x_{i} \, y, &\\
x_{l} \longmapsto x_{l},~~~   l\neq i,i+1, \\
y \longmapsto y,
\end{array} \right.
$$
for all $i = 1, 2, \ldots, n-1.$
\end{thm}

\medskip

This  representation was independently  considered in  \cite{M}.

Remark that the group $WB_{n}$ can be considered as a quotient of  $\psi(VB_n)$: in fact a straightforward verification shows that:

 \begin{prop}\label{virtualvswelded}
 Let $q_n: VB_n \to  WB_n$ be the  projection defined by $q_n(\sigma_i)=\sigma_i$ and
  $q_n(\rho_i)=\alpha_i$ for $i=1, \ldots, n$.
 Let $F_{n+1} = \langle x_1, x_2, \ldots, x_n, y \rangle$  and $F_{n} =  \langle x_1, x_2, \ldots, x_n \rangle$.
 The projection $p_n: F_{n+1} \to F_{n+1} / \langle\langle y \rangle  \rangle\simeq F_n$ induces a map
$p_n^\# : \psi(VB_n) \to  WB_{n}$ such that  $p_n^\# \circ \psi= q_n$.
 \end{prop}

The faithfulness of the representation given in Theorem \ref{theorem3}
is evident  for $n=2$ since in this case $VB_2 \simeq C_2 \simeq \mathbb{Z} * \mathbb{Z}_2$.
In fact,  from the defining relations it follows  that $VB_2 \simeq C_2$ and if we consider composition of $\psi$ with $p_2 : F_3 = \langle x_1, x_2, y \rangle \to F_2 = \langle x_1, x_2 \rangle$ we get the representation of $C_2$ by
automorphisms of $F_2.$

For $n>2$  we do not know if above representation is faithful:
since
 $\psi(VB_n) \subseteq WB_{n+1}$ the faithfulness of $\psi$ for any $n$ would imply that virtual braid groups can be considered as subgroup of welded groups, whose structure and applications in finite type invariants theory is much more advanced (see for instance \cite{Bn,BerP}).

 We remark also that a quite tedious computation shows that  image by $\psi$ of the \emph{Kishino braid} :
  $Kb= \sigma_2 \sigma_1 \rho_2 \sigma_1^{-1} \sigma_2^{-1}
 \rho_1   \sigma_2^{-1}  \sigma_1^{-1} \rho_2 \sigma_1 \sigma_2 \sigma_2 \sigma_1 \rho_2 \sigma_1^{-1} \sigma_2^{-1}
 \rho_1   \sigma_2^{-1}  \sigma_1^{-1} \rho_2 \sigma_1 \sigma_2$
  is non trivial while
its Alexander invariant is trivial  \cite{Bnweb}.

Notice that
$Kb= \sigma_2^{-1} b_1^2 \sigma_2$ where
$b_1 = \sigma_2^2 \sigma_1 \rho_2 \sigma_1^{-1} \sigma_2^{-1} \rho_1 \sigma_2^{-1} $ $\sigma_1^{-1} \rho_2 \sigma_1$.


\section{Groups of virtual links}\label{gvl}

In the classical case the group of a link $L$ is defined as the fundamental group $\pi_1(S^3 \setminus N(L))$ where
$N(L)$ is
a tubular neighborhood of the link in $S^3$.  To find a group presentation of this group  we can use Wirtinger
method as follows.

One can consider the oriented diagram of the link as the union of  oriented arcs in the plane.
Define a base point  for $\pi_1(S^3 \setminus N(L))$ and associate to any arc
a loop starting from the base point, which  goes straight to the chosen arc, encircles it with  linking number $+1$
and returns straight to the base point. Let us consider the loops $a_i, a_j, a_k$ around three arcs in a crossing of the diagram as in Figure \ref{wirtingerfigure}:

\begin{figure}[h]
\centering
 \includegraphics[width=6cm,]{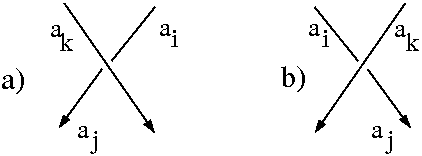}
\caption{Arcs around two types of crossings} \label{wirtingerfigure}
\end{figure}

One can easily verify that in the first case the loop $a_j$ is homotopic to $a_k a_i a_k^{-1}$ and in the second case the
loop $a_j$ is homotopic to $a_k^{-1} a_i a_k$.

The group  $\pi_1(S^3 \setminus N(L))$ admits the following presentation:
Let $D$ be an oriented diagram of a 
a link $L$
and  $A_1, \ldots, A_n$ be the arcs determined by $D$. The group $\pi_1(S^3 \setminus N(L))$, admits the following group presentation:

\noindent \textbf{Generators:} $\{a_{1}, \ldots a_{n} \}$, where $a_j$ is the loop associated to the arc $A_j$.

\noindent \textbf{Relations:}
To each crossing corresponds a relation as follows

\begin{eqnarray*}
&a_j a_k= a_k a_i &\text{ if } a_i, a_j, a_k \, \text{meet in a crossing like in case a) of  figure \ref{wirtingerfigure}};\\
&a_k a_j= a_i a_k &\text{ if } a_i, a_j, a_k \, \text{meet in a crossing like in case b) of  figure \ref{wirtingerfigure}}.
\end{eqnarray*}

We  recall that this presentation  is usually called upper Wirtinger presentation while the lower  Wirtinger presentation
is  obtained  applying Wirtinger method to the diagram where  all crossings are reversed;
  these presentations are generally different but the corresponding groups are evidently  isomorphic because of their geometrical meaning.

Another way to obtain a group presentation for $
\pi_1(S^3 \setminus N(L))
$ is to consider a braid
$\beta \in B_n$ such that its Alexander closure is isotopic to $L$; therefore the group
$
\pi_1(S^3 \setminus N(L))
$
admits the presentation:
$$ \pi_1(S^3 \setminus N(L))= \langle x_1, x_2, \ldots , x_n~ ||~ x_i = \beta (x_i),~~i=1, \ldots, n \rangle \, ,$$
where we consider  $\beta$ as an automorphism of $F_n$    (this a consequence of van Kampen's Theorem, see for instance
 \cite{W}).

Given  a  virtual link $vL$,
 according to  \cite{Ka}  the group of the virtual knot $vL$, denoted with $G_{K,v}(vL)$, is the group obtained extending the Wirtinger method to virtual diagrams, forgetting all virtual crossings.   This notion of group of a virtual knot (or link)  does not seem satisfactory: for instance
if $vT$ is the virtual trefoil knot with two classical crossings and one virtual crossing and
$U$ is unknot then $G_{K,v}(vT) \simeq G_{K,v}(U) \simeq \mathbb{Z}$, although that $vT$
is not equivalent to $U$. In addition, as noted Goussarov-Polyak-Viro \cite{GPV},   the upper Wirtinger group  of a virtual knot
is not necessary isomorphic
to the corresponding  lower Wirtinger group. 

We introduce another notion of group $G_v(vL)$ of a virtual link $vL$. Let $vL = \widehat{\beta_v}$ be a
closed virtual braid, where $\beta_v \in VB_n$ \footnote{The indices $v$ and $w$ are given to precise when we are considering virtual or welded braids}. Define
$$
G_v(vL) = \langle x_1, x_2, \ldots , x_n, y~ ||~ x_i = \psi(\beta_v)(x_i),~~i=1, \ldots, n \rangle.
$$
We will consider the action from left to right and for simplicity of notation we will write $x_i \beta_v$ instead of $\psi(\beta_v)(x_i).$

\medskip

\noindent {\bf Notation.} In the following we use the notations $[a, b]= a^{-1} b^{-1} a b$
and $a^b=b^{-1} a b$.

The following Theorem was announced in \cite{Bar-1}.

\begin{thm}  \label{theorem4}
 The group $G_v(vL)$ is an invariant of the virtual link $vL$.
 \end{thm}
  \begin{proof}
According to   \cite{K} two virtual braids have equivalent closures as virtual links if and only
if they are related by a finite sequence of the following moves:

1) a braid move (which is a move corresponding to a defining relation of the virtual braid group),

2) a conjugation in the virtual braid group,

3) a right stabilization of positive, negative or virtual type, and its inverse operation,

4) a right/left virtual exchange move.

Here \emph{a right stabilization of positive, negative or virtual type} is a replacement of $b \in VB_n$ by
$b \sigma_n,$ $b \sigma_n^{-1}$ or $b \rho_n$ $\in VB_{n+1},$ respectively, \emph{a right virtual exchange move} is a replacement
$$
b_1 \sigma_n^{-1} b_2 \sigma_n \longleftrightarrow b_1 \rho_n b_2 \rho_n \in VB_{n+1}
$$
and
\emph{a left virtual exchange move} is a replacement
$$
s(b_1) \sigma_1^{-1} s(b_2) \sigma_1 \longleftrightarrow s(b_1) \rho_1 s(b_2) \rho_1 \in VB_{n+1},
$$
where $b_1, b_2 \in VB_n$ and $s : VB_n \longrightarrow VB_{n+1}$ is the shift map i.e. $s(\sigma_i) = \sigma_{i+1}.$

We have to check that under all  moves 1) - 4), the group $G_v(vL)$ does not change. Let $vL = \widehat{\beta_v}$ where $\beta_v \in VB_n.$

1) If $\beta_v' \in VB_n$ is another braid such that $\beta_v = \beta_v'$ in $VB_n$, then $\psi(\beta_v) = \psi(\beta_v')$ since $\psi$ is a homomorphism.
Hence $G(\widehat{\beta_v}) = G(\widehat{\beta_v'})$ and the first move does not change the group of virtual link.

2) Evidently it is enough to consider only conjugations by the generators of $VB_n$. Let
$$
G_1 = G(\widehat{\beta_v}) = \langle x_1, x_2, \ldots, x_n, y ~||~x_i = x_i \beta_v,~i = 1, 2, \ldots, n \rangle
$$
and
$$
G_2 = G(\widehat{\sigma_k \beta_v \sigma_k^{-1}}) = \langle x_1, x_2, \ldots, x_n, y ~||~x_i = x_i (\sigma_k \beta_v \sigma_k^{-1}),~i = 1, 2,
\ldots, n \rangle,
$$
where $k \in \{ 1, 2, \ldots, n-1 \}.$
To prove that $G_2 \simeq G_1$ we rewrite the defining relations of $G_2$ in the form
$$
x_i \sigma_k = x_i (\sigma_k  \beta_v),~i = 1, 2, \ldots, n.
$$
If $i \not= k, k-1$ then this relation is equivalent to
$$
 x_i = x_i \beta_v
$$
since $x_i \sigma_k = x_i.$
But it is a relation in $G_1.$ Hence, we have to consider only two relations:
$$
x_k \sigma_k = x_k (\sigma_k \beta_v),~~~x_{k+1} \sigma_k = x_{k+1} (\sigma_k \beta_v).
$$
By the definition of $\psi$ these relations are equivalent to
$$
x_k x_{k+1} x_k^{-1} = (x_k x_{k+1} x_k^{-1}) \beta_v,~~~x_{k} = x_{k} \beta_v.
$$
The second relation is a relation from $G_1$. Rewrite the first relation:
$$
x_k x_{k+1} x_k^{-1} = (x_k \beta_v)  (x_{k+1} \beta_v) (x_k^{-1} \beta_v)
$$
and using the second relation we get
$$
 x_{k+1} = x_{k+1} \beta_v,
$$
which is a relation from $G_1$. Hence we proved that any relation from $G_1$ is true in $G_2$ and analogously one can prove
that any relation from $G_2$ is true in $G_1.$

Consider the conjugation by element $\rho_k.$ In this case we have
$$
G_2 = G(\widehat{\rho_k \beta_v \rho_k}) = \langle x_1, x_2, \ldots, x_n, y ~||~x_i = x_i (\rho_k \beta_v \rho_k),~i = 1, 2,
\ldots, n \rangle.
$$
Rewrite the relations of $G_2$ in the form
$$
x_i \rho_k = x_i (\rho_k \beta_v ),~i = 1, 2, \ldots, n.
$$
If $i \not= k, k-1$ then we have
$$
 x_i = x_i \beta_v,
$$
since $x_i \rho_k = x_i$ But it is a relation in $G_1.$ Hence, we have to consider only two relations:
$$
x_k \rho_k = x_k (\rho_k \beta_v),~~~x_{k+1} \rho_k = x_{k+1} (\rho_k \beta_v).
$$
By the definition of $\psi$ these relations are equivalent to
$$
y x_{k+1} y^{-1} = (y x_{k+1} y^{-1}) \beta_v,~~~y^{-1} x_{k} y = ( y^{-1} x_{k} y) \beta_v
$$
or
$$
y x_{k+1} y^{-1} = y (x_{k+1} \beta_v) y^{-1},~~~y^{-1} x_{k} y = y^{-1} (x_{k} \beta_v) y.
$$
These relations are equivalent to
$$
x_{k+1} = x_{k+1} \beta_v,~~~x_{k} = x_{k} \beta_v
$$
which are relations from $G_1.$ Hence we proved that the set of relations from $G_2$ is equivalent to the set of relations from $G_1$.

3) Consider the move from a braid $\beta_v = \beta_v(\sigma_1, \sigma_2, \ldots, \sigma_{n-1}, \rho_1, \rho_2, \ldots, \rho_{n-1}) \in VB_n$
to the braid $\beta_v \sigma_n^{-1} \in VB_{n+1}.$
We have two groups:
$$
G_1 = G(\widehat{\beta_v}) = \langle x_1, x_2, \ldots, x_n, y ~||~x_i = x_i \beta_v,~i = 1, 2, \ldots, n \rangle
$$
and
$$
G_2 = G(\widehat{\beta_v \sigma_n^{-1}}) = \langle x_1, x_2, \ldots, x_n, x_{n+1}, y ~||~x_i = x_i (\beta_v \sigma_n^{-1}),~i = 1, 2,
\ldots, n+1 \rangle
$$
and we need to  prove that they are isomorphic.

Rewrite the relations of $G_2$ in the form
$$
x_i \sigma_n = x_i \beta_v,~i = 1, 2, \ldots, n+1.
$$
If $i = 1, 2, \ldots, n-1$ then we have
$$
 x_i = x_i \beta_v,
$$
which is a relation in $G_1.$ Hence, we have to consider only two relations:
$$
x_n \sigma_n = x_n \beta_v,~~~x_{n+1} \sigma_n = x_{n+1} \beta_v,
$$
these relations are equivalent to
$$
x_n x_{n+1} x_n^{-1} = x_{n} \beta_v,~~~x_{n} = x_{n+1}.
$$
Using the second relation rewrite the first in the form
$$
x_{n} =  x_{n} \beta_v,
$$
which  is a relation from $G_1.$ Also we can remove $x_{n+1}$ from the set of generators of $G_2$. Hence we proved that the set of relations from $G_2$ is equivalent to the set of relations from $G_1$.

The move from a braid $\beta_v \in VB_n$ to the braid $\beta_v \sigma_n$ is similar.

Consider the move from a braid $\beta_v  \in VB_n$ to the braid $\beta_v \rho_n \in VB_{n+1}.$
We have two groups:
$$
G_1 = G(\widehat{\beta_v}) = \langle x_1, x_2, \ldots, x_n, y ~||~x_i = x_i \beta_v,~i = 1, 2, \ldots, n \rangle
$$
and
$$
G_2 = G(\widehat{\beta_v \rho_n}) = \langle x_1, x_2, \ldots, x_n, x_{n+1}, y ~||~x_i = x_i (\beta_v \rho_n),~i = 1, 2,
\ldots, n+1 \rangle .
$$

For $i = n$ and $i = n+1$ we have the following relations in $G_2$:
$$
x_n \rho_n = x_n \beta_v,~~~x_{n+1} \rho_n = x_{n+1} \beta_v,
$$
which are equivalent to
$$
y x_{n+1} y^{-1} = x_n \beta_v,~~~y^{-1} x_{n} y = x_{n+1}.
$$
Rewrite the second relation in the form $x_{n} = y x_{n+1} y^{-1}$ and substituting in the first relation we have
$$
x_n  = x_{n} \beta_v,
$$
which is a relation from $G_1.$ Also we can remove $x_{n+1}$ from the set of generators of $G_2$. Hence we proved that the set of relations from $G_2$ is equivalent to the set of relations from $G_1$.

4) Finally, consider the exchange  move \footnote{In this formula we take $b_2^{-1}$ instead $b_2$ for convenience.}
$$
b_1 \sigma_n^{-1} b_2^{-1} \sigma_n \longleftrightarrow b_1 \rho_n b_2^{-1} \rho_n,~~~b_1, b_2 \in VB_n.
$$

We have two groups:
$$
G_1 = G(\widehat{b_1 \sigma_n^{-1} b_2^{-1} \sigma_n}) = \langle x_1, x_2, \ldots, x_{n+1}, y ~||~x_i = x_i (b_1 \sigma_n^{-1} b_2^{-1} \sigma_n),
~i = 1, 2, \ldots, n+1 \rangle
$$
and
$$
G_2 = G(\widehat{b_1 \rho_n b_2^{-1} \rho_n}) = \langle x_1, x_2, \ldots, x_{n+1}, y ~||~x_i = x_i (b_1 \rho_n b_2^{-1} \rho_n),~i = 1, 2,
\ldots, n+1 \rangle
$$
and we need to prove that they are isomorphic.

Rewrite the defining relations from $G_1$ in the form
$$
x_i (\sigma_n^{-1} b_2) = x_i (b_1 \sigma_n^{-1}),~i = 1, 2, \ldots, n+1,
$$
and defining relations from $G_2$ in the form
$$
x_i (\rho_n b_2) = x_i (b_1 \rho_n),~i = 1, 2, \ldots, n+1.
$$
Let $b_1$ and $b_2$ are the following automorphisms
$$
b_1 :
\left\{
  \begin{array}{l}
x_1 \longmapsto x_{\pi(1)}^{a_1}, \\
\\
x_2 \longmapsto x_{\pi(2)}^{a_2}, \\
\\
................... \\
\\
x_n \longmapsto x_{\pi(n)}^{a_n}, \\
\\
x_{n+1} \longmapsto x_{n+1}, \\
\\
y \longmapsto y, \\
  \end{array}
\right.~~~~
b_2 :
\left\{
  \begin{array}{l}
x_1 \longmapsto x_{\tau(1)}^{c_1}, \\
\\
x_2 \longmapsto x_{\tau(2)}^{c_2}, \\
\\
................... \\
\\
x_n \longmapsto x_{\tau(n)}^{c_n}, \\
\\
x_{n+1} \longmapsto x_{n+1}, \\
\\
y \longmapsto y, \\
  \end{array}
\right.,
$$
where $\pi, \tau \in S_n$ and $a_i, c_i \in \langle x_1, x_2, \ldots, x_n, y \rangle$. Then the automorphism $\sigma_n^{-1} b_2$ has the form
$$
\sigma_n^{-1} b_2 :
\left\{
  \begin{array}{l}
x_1 \longmapsto x_{\tau(1)}^{c_1}, \\
\\
x_2 \longmapsto x_{\tau(2)}^{c_2}, \\
\\
................... \\
\\
x_{n-1} \longmapsto x_{\tau(n-1)}^{c_{n-1}}, \\
\\
x_{n} \longmapsto x_{n+1}, \\
\\
x_{n+1} \longmapsto x_{n+1}^{-1} x_{\tau(n)}^{c_n} x_{n+1}, \\
\\
y \longmapsto y, \\
  \end{array}
\right.
$$
and  the automorphism $b_1 \sigma_n^{-1}$ has the form
$$
b_1 \sigma_n^{-1} :
\left\{
  \begin{array}{l}
x_1 \longmapsto (x_{\pi(1)} \sigma_n^{-1})^{a_1 \sigma_n^{-1}}, \\
\\
x_2 \longmapsto (x_{\pi(2)} \sigma_n^{-1})^{a_2 \sigma_n^{-1}}, \\
\\
................................ \\
\\
x_{n-1} \longmapsto (x_{\pi(n-1)} \sigma_n^{-1})^{a_{n-1} \sigma_n^{-1}}, \\
\\
x_n \longmapsto (x_{\pi(n)} \sigma_n^{-1})^{a_n \sigma_n^{-1}}, \\
\\
x_{n+1} \longmapsto x_{n+1}^{-1} x_{n} x_{n+1}, \\
\\
y \longmapsto y. \\
  \end{array}
\right.
$$
Hence the first group has the following presentation
$$
G_1 = \langle x_1, x_2, \ldots, x_{n+1}, y ~||~x_{\tau(1)}^{c_1} = (x_{\pi(1)} \sigma_n^{-1})^{a_1 \sigma_n^{-1}},~~
x_{\tau(2)}^{c_2} = (x_{\pi(2)} \sigma_n^{-1})^{a_2 \sigma_n^{-1}}, \ldots,
$$
$$
x_{\tau(n-1)}^{c_{n-1}} = (x_{\pi(n-1)} \sigma_n^{-1})^{a_{n-1} \sigma_n^{-1}},~~
x_{n+1} = (x_{\pi(n)} \sigma_n^{-1})^{a_n \sigma_n^{-1}},~~
x_{\tau(n)}^{c_n} = x_n \rangle.
$$

Analogously, construct the presentation for $G_2.$ Calculate the automorphism
$$
\rho_n b_2 :
\left\{
  \begin{array}{l}
x_1 \longmapsto x_{\tau(1)}^{c_1}, \\
\\
x_2 \longmapsto x_{\tau(2)}^{c_2}, \\
\\
.................... \\
\\
x_{n-1} \longmapsto x_{\tau(n-1)}^{c_{n-1}}, \\
\\
x_{n} \longmapsto y x_{n+1} y^{-1}, \\
\\
x_{n+1} \longmapsto y^{-1} x_{\tau(n)}^{c_n} y, \\
\\
y \longmapsto y, \\
  \end{array}
\right.
$$
and  the automorphism $b_1 \rho_n$ has the form
$$
b_1 \rho_n :
\left\{
  \begin{array}{l}
x_1 \longmapsto (x_{\pi(1)} \rho_n)^{a_1 \rho_n}, \\
\\
x_2 \longmapsto (x_{\pi(2)} \rho_n)^{a_2 \rho_n}, \\
\\
................................ \\
\\
x_{n-1} \longmapsto (x_{\pi(n-1)} \rho_n)^{a_{n-1} \rho_n}, \\
\\
x_n \longmapsto (x_{\pi(n)} \rho_n)^{a_n \rho_n}, \\
\\
x_{n+1} \longmapsto y^{-1} x_{n} y, \\
\\
y \longmapsto y. \\
  \end{array}
\right.
$$
Hence the second group has the following presentation
$$
G_2 = \langle x_1, x_2, \ldots, x_{n+1}, y ~||~x_{\tau(1)}^{c_1} = (x_{\pi(1)} \rho_n)^{a_1 \rho_n},~~
x_{\tau(2)}^{c_2} = (x_{\pi(2)} \rho_n)^{a_2 \rho_n}, \ldots,
$$
$$
x_{\tau(n-1)}^{c_{n-1}} = (x_{\pi(n-1)} \rho_n)^{a_{n-1} \rho_n},~~
y x_{n+1} y^{-1} = (x_{\pi(n)} \rho_n)^{a_n \rho_n},~~
x_{\tau(n)}^{c_n} = x_n \rangle.
$$
Compare $G_1$ and $G_2$:
since $a_i = a_i(x_1, x_2, \ldots, x_n, y),$ let us  denote
$$
a'_i = a_i \sigma_n^{-1} = a_i(x_1, x_2, \ldots, x_{n-1}, x_{n+1}, y)
$$
and
$$
a''_i = a_i \rho_n = a_i(x_1, x_2, \ldots, x_{n-1}, y x_{n+1} y^{-1}, y).
$$
Then
$$
G_1 = \langle x_1, x_2, \ldots, x_{n+1}, y ~||~x_{\tau(1)}^{c_1} = (x_{\pi(1)} \sigma_n^{-1})^{a'_1},~~
x_{\tau(2)}^{c_2} = (x_{\pi(2)} \sigma_n^{-1})^{a'_2}, \ldots,
$$
$$
x_{\tau(n-1)}^{c_{n-1}} = (x_{\pi(n-1)} \sigma_n^{-1})^{a'_{n-1}},~~
x_{n+1} = (x_{\pi(n)} \sigma_n^{-1})^{a'_n},~~
x_{\tau(n)}^{c_n} = x_n \rangle.
$$
and
$$
G_2 = \langle x_1, x_2, \ldots, x_{n+1}, y ~||~x_{\tau(1)}^{c_1} = (x_{\pi(1)} \rho_n)^{a''_1},~~
x_{\tau(2)}^{c_2} = (x_{\pi(2)} \rho_n)^{a''_2}, \ldots,
$$
$$
x_{\tau(n-1)}^{c_{n-1}} = (x_{\pi(n-1)} \rho_n)^{a''_{n-1}},~~
y x_{n+1} y^{-1} = (x_{\pi(n)} \rho_n)^{a''_n},~~
x_{\tau(n)}^{c_n} = x_n \rangle.
$$
Denote by $z_{n+1} = y x_{n+1} y^{-1}$;  the group  $G_2$ has therefore the  following presentation
$$
G_2 = \langle x_1, x_2, \ldots, x_{n}, z_{n+1}, y ~||~x_{\tau(1)}^{c_1} = (x_{\pi(1)} \rho_n)^{a'_1},~~
x_{\tau(2)}^{c_2} = (x_{\pi(2)} \rho_n)^{a'_2}, \ldots,
$$
$$
x_{\tau(n-1)}^{c_{n-1}} = (x_{\pi(n-1)} \rho_n)^{a'_{n-1}},~~
z_{n+1} = (x_{\pi(n)} \rho_n)^{a'_n},~~x_{\tau(n)}^{c_n} = x_n \rangle,
$$
where $a'_i = a_i \sigma_n^{-1} = a_i(x_1, x_2, \ldots, x_{n-1}, z_{n+1}, y).$
We will assume that $n = \pi(1)$ (other cases consider analogously). Then our groups have presentations:
$$
G_1 = \langle x_1, x_2, \ldots, x_{n+1}, y ~||~x_{\tau(1)}^{c_1} = x_{n+1}^{a'_1},~~
x_{\tau(2)}^{c_2} = (x_{\pi(2)})^{a'_2}, \ldots,
$$
$$
x_{\tau(n-1)}^{c_{n-1}} = (x_{\pi(n-1)})^{a'_{n-1}},~~
x_{n+1} = (x_{\pi(n)})^{a'_n},~~
x_{\tau(n)}^{c_n} = x_n \rangle,
$$
$$
G_2 = \langle x_1, x_2, \ldots, x_{n}, z_{n+1}, y ~||~x_{\tau(1)}^{c_1} = (z_{n+1})^{a'_1},~~
x_{\tau(2)}^{c_2} = (x_{\pi(2)})^{a'_2}, \ldots,
$$
$$
x_{\tau(n-1)}^{c_{n-1}} = (x_{\pi(n-1)})^{a'_{n-1}},~~
z_{n+1} = (x_{\pi(n)})^{a'_n},~~x_{\tau(n)}^{c_n} = x_n \rangle.
$$
and therefore they are isomorphic.
\end{proof}

{\bf Example 0.}  For the unknot $U$ we have
$$
G_v(U) = \langle x, y \rangle \simeq F_2.
$$

\medskip

{\bf Example 1.} For  the virtual trefoil $vT$ we have that $vT = \widehat{ \sigma_1^2 \rho_1}$  and then
$$
G_v(vT) = \langle x, y ~||~ x \, (y \, x \, y^{-2} \, x \, y) = (y \, x \, y^{-2} \, x \, y) \, x
\rangle \not\simeq F_2.
$$
Hence, we obtain a new proof of the fact that $vT$  a non-trivial virtual knot.

\medskip

{\bf Example 2.}  The group $G_v(K)$ is not a complete invariant for virtual knots.  Let $c = \rho_1 \sigma_1 \sigma_2 \sigma_1 \rho_1 \sigma_1^{-1} \sigma_2^{-1} \sigma_1^{-1} \in VB_3$; the closure $\widehat{c}$ is
equivalent to the Kishino knot (see Figure \ref{kishino}). The Kishino knot 
 is a non trivial virtual knot \cite{ESSSW} with trivial Jones polynomial and trivial fundamental group ($G_{K,v}( \widehat{c} ) =\Z$).
   For this knot we have that $G_v(\widehat{c}) = F_2$.

\begin{figure}[h]
\centering
 \includegraphics[width=10cm,]{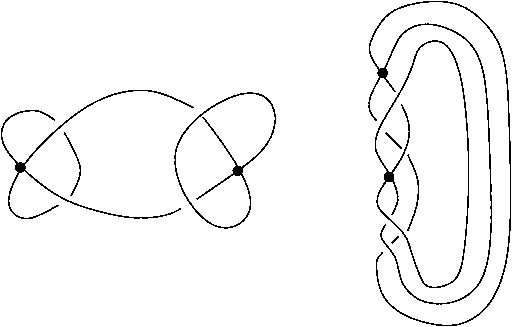}
\caption{The Kishino knot: usual diagram and as  the closure of a virtual braid.} \label{kishino}
\end{figure}

In fact, it is not difficult to prove that $c$ defines the following automorphism of $F_4$
$$
\psi(c) :
\left\{
  \begin{array}{l}
x_1 \longmapsto y^2 x_3^{-1} x_2 x_3 y^{-2} x_3 y^2 x_3^{-1} x_2^{-1} x_3 y^{-2}, \\
\\
x_2 \longmapsto x_3^{-1} x_2 x_3 y^{-2} x_3 y x_3^{-1} x_2^{-1} x_1 x_2 x_3 y^{-1} x_3^{-1} y^2 x_3^{-1} x_2^{-1} x_3, \\
\\
x_{3} \longmapsto y x_3^{-1} x_2 x_3 y^{-1}, \\
\\
y \longmapsto y, \\
  \end{array}
\right.
$$
and therefore the  image of $c$ as automorphism of $F_4$ is non trivial: nevertheless we have that

$$
G_v(\widehat{c}) = \langle x_1, x_2, x_3, y~||~x_1 = y^2 x_3^{-1} x_2 x_3 y^{-2} x_3 y^2 x_3^{-1} x_2^{-1} x_3 y^{-2},
$$
$$
x_2 = x_3^{-1} x_2 x_3 y^{-2} x_3 y x_3^{-1} x_2^{-1} x_1 x_2 x_3 y^{-1} x_3^{-1} y^2 x_3^{-1} x_2^{-1} x_3, \;
x_3 = y x_3^{-1} x_2 x_3 y^{-1}
\rangle.
$$
Using the first relation we can remove $x_1$ and using the third relation we can remove $x_2$. We get
$$
G_v(\widehat{c}) = \langle x, y~||~x y^{-1} x y x^{-1} =  x y^{-1} x y x^{-1} \rangle \simeq F_2,
$$
where $x = x_3$

\medskip

{\bf Example 3.} Let $b_1 = b_2^{-1} = \sigma_1 \rho_1 \sigma_1$ and $b = b_1 \rho_2 b_2 \rho_2.$ It is easy to check that for this virtual braid its group
$$
G_v( \widehat{b} ) = \langle x_1, x_2, y ~||~y x_1 y^{-1} = x_2\rangle
$$
is an HNN-extension of the free group $\langle x_1, x_2 \rangle$ with cyclic associated subgroups.
Therefore the closure of $b$ is a non trivial virtual link.

\medskip

The following Propositions establish the relations between the different notions of groups of (virtual) knots.

\begin{prop}\label{propos}
Let $K$ be a classical knot then
$$
G_v(K)=\mathbb{Z}\ast \pi_1(S^3 \setminus N(K)),~~~~\mathbb{Z}= \langle y \rangle.
$$
\end{prop}
\begin{proof}
As previously recalled, if $\beta \in B_n$ is a braid such that its Alexander closure is isotopic to $K$, the group
$
\pi_1(S^3 \setminus N(K))
$ admits
the presentation
$ \langle x_1, x_2, \ldots , x_n~ ||~ x_i = \beta (x_i),~~i=1, \ldots, n \rangle$.

The claim follows therefore from the remark that the representation $\psi$ of $VB_n$ in $\mbox{Aut}(F_{n+1})$ restraint to $B_n$ coincides with the usual Artin representation
composed with the natural inclusion $ \iota: \mbox{Aut}(F_{n}) \to \mbox{Aut}(F_{n+1})$.
\end{proof}

The Proposition below will be proved at the end of Section \ref{gwl}.

\begin{prop}\label{isomprincipal}
 Let $vK$ be a virtual knot.
 \begin{enumerate}
 \item  The group
$G_v(vK)/\langle \langle y \rangle \rangle$ is isomorphic to $G_{K,v}(vK)$;
\item The abelianization of $G_v(vK)$ is isomorphic to $\Z^2$.
\end{enumerate}
\end{prop}

We recall that a group $G$ is residually finite if for any nontrivial element $g \in G$
there exists a finite-index subgroup of $G$ which does not   contain $g$.

 According to \cite{He} every knot group
of a classical knot  is residually finite, while the  virtual knot
groups defined by Kauffman does not need to have this property (see \cite{SW}).

\begin{prop}
 Let $K$ be a classical knot then
$G_v(K)$ is residually finite.
\end{prop}
\begin{proof}
The free product of two residually finite groups is residually finite \cite{Gr}.
\end{proof}

We do not  know whether $G_v(vK)$ is residually finite for any virtual knot $vK$.

\begin{rem}
We can modify Wirtinger method to virtual links using previous representation of $VB_n$. 
Let us now consider each  virtual crossing as the common endpoint of  four
different arcs 
and  mark  the  obtained arcs with labels $x_1,\ldots, x_m$ and  add an element $y$ of this set. Now consider the group $G_y(vK)$ generated by elements $x_1, \ldots, x_m,y$ under  the usual Wirtinger relations for classical crossings  plus
the following relations for virtual crossing: the labeling of arcs $x_k, x_l, x_i, x_j $ meeting in a virtual crossing   as in Figure \ref{wirtweld}
respect the relations $x_i=y^{-1} x_l y$ and $x_j=y x_k y^{-1}$.
One can easily check that $G_y(vK)$  is actually invariant under 
virtual and mixed Reidemeister moves and hence  is an invariant for virtual knots.
Arguments in Proposition \ref{isomweld} can be therefore easily adapted to this case to prove that  $G_y(vK)$  is   isomorphic  $G_{v}(vK)$.
\end{rem}

\begin{rem}
Using the "Wirtinger like" labeling proposed in previous  Remark
it is also possible to extend the notion of group system to  $G_{v}(vK)$.
We recall that the group  system of a classical knot $K$ is given by the knot group, a meridian and its corresponding longitude: in the case of $G_{v}(vK)$
we can call meridian the generator corresponding to any arc. The longitude corresponding to this meridian is defined as follows: we go along the diagram starting from this arc (say $a_i$) and we write $a_k^{-1}$ when passing under $a_k$ as in  Figure \ref{wirtingerfigure} a) and $a_k$ when passing under $a_k$ as in  Figure \ref{wirtingerfigure} b).  On the other hand if  we encounter a virtual crossing according Figure \ref{wirtweld}  we write $y$ when we pass from $x_l$ to $x_i$ and we write $y^{-1}$ when we pass from $x_k$  to $x_j$. Finally we write  $a_i^{-m}$ where $m$ is the length (the sum of exponents) of the word that we wrote following the diagram. It is easy to verify that such an element belongs to the commutator subgroup of    $G_{v}(vK)$ and that  meridian and longitude are well defined under generalized Reidemeister moves.
\end{rem}

\section{Groups of welded links}\label{gwl}

As in the case of virtual links, Wirtinger method can be naturally adapted also to welded links: it suffices to check  that
also the forbidden relation $F1$ is preserved by Wirtinger labeling.
  Given  a  welded link $wL$ we can therefore define  the group of the welded link $wL$, $G_{K,w}(wL)$, as the group obtained extending the Wirtinger method to welded diagrams, forgetting all welded crossings.
  This fact has been  already remarked: see for instance  \cite{Wi}, where  the notion of {\em group system} is extended to welded knot diagrams.

Notice that the forbidden move $F2$ is not preserved by above method and then that the lower Wirtinger presentation does not extend to welded link diagrams.

On the other hand, considering welded links as closure of welded braids we can deduce as in the virtual case another possible definition of group of a welded link.  Let $wL = \widehat{\beta_w}$ be the closure
of the welded braid  $\beta_w \in WB_n$. Define
$$
G_w(wL) = \langle x_1, x_2, \ldots , x_n~ ||~ x_i = \beta_w(x_i),~~i=1, \ldots, n \rangle.
$$

\begin{thm}  \label{theorem5}
 The group $G_w(wL)$ is an invariant of the virtual link $wL$.
 \end{thm}
 \begin{proof}
We recall that two welded  braids have equivalent closures as welded links if and only
if they are related by a finite sequence of the following moves \cite{K}:

1) a braid move (which is a move corresponding to a defining relation of the welded  braid group),

2) a conjugation in the welded braid group,

3) a right stabilization of positive, negative or welded type, and its inverse operation,

We have to check that under all  moves 1) - 3), the group $G_w(wL)$ does not change: the move 1) is evident and
for the moves 2) and 3) we can repeat verbatim the proof of Theorem~\ref{theorem4}.
 \end{proof}

\begin{prop}\label{isomweld}
Let $wK$ be a welded knot.  The groups  $G_w(wK)$ and $G_{K,w}(wK)$ are isomorphic.
\end{prop}
\begin{proof}
Any generator of $WB_n$ acts trivially on generators of $F_n$ except a pair of  generators:
more precisely we have respectively that

\begin{eqnarray*} 
x_i \cdot \sigma_i= x_i x_{i+1} x_i^{-1} &:=u^+(x_i, x_{i+1})  \, ,\\
x_{i+1} \cdot \sigma_i= x_i &:=v^+(x_i, x_{i+1}) \, ,\\
x_j \cdot \sigma_i= x_j &\qquad j \not= i, i+1 \, ,
\end{eqnarray*}

\begin{eqnarray*} 
x_i \cdot \sigma_i^{-1}=  x_{i+1}  &:=u^-(x_i, x_{i+1}) \, ,\\
x_{i+1} \cdot \sigma_i^{-1}= x_{i+1}^{-1}  x_i x_{i+1} &:=v^-(x_i, x_{i+1}) \, ,\\
x_j \cdot \sigma_i^{-1}= x_j &\qquad j \not= i, i+1 \, ,
\end{eqnarray*}

\begin{eqnarray*} 
x_i \cdot \alpha_i=   x_{i+1}  &:=u^\circ(x_i, x_{i+1}) \, ,\\
x_{i+1} \cdot \alpha_i= x_i &:=v^\circ(x_i, x_{i+1}) \, ,\\
x_j \cdot \alpha_i= x_j &\qquad j \not= i, i+1 \, .
\end{eqnarray*}

Let $\beta_w$ a welded braid with closure equivalent to $wK$.
Let us regard to the diagram representing the closure of
$\beta_w$,  $\widehat{\beta_w}$, as a directed graph and denote the edges of $\widehat{\beta_w}$, the generators of $G_{K,w}(wK)$ by labels $x_1, \ldots, x_m$: according to Wirtinger method recalled in Section 4,  for each crossing of $\widehat{\beta_w}$ (see Figure \ref{wirtweld})
we have the following relations:

\begin{enumerate}
\item If the crossing is positive  $x_i=x_l$ and
$x_j= x_{l}^{-1}  x_k x_{l} $;
\item If the crossing is negative  $x_i=x_k x_{l} x_k^{-1}$ and
$x_j=x_k$ where the word $u^+(x_k, x_l)$ and $v^+(x_k, x_l)$ are the words defined above;
\item If the crossing is welded  $x_i=x_l$ and
$x_j=x_k$  are the words defined above.
\end{enumerate}

Now let us remark that  we have  that $x_l=u^-(x_k, x_l)$ , $x_{l}^{-1}  x_k x_{l} =v^-(x_k, x_l)$, 
$x_k x_{l} x_k^{-1}=u^+(x_k, x_l)$, $x_k=v^+(x_k, x_l)$, $x_l=u^\circ(x_k, x_l)$ and $x_k=v^\circ(x_k, x_l)$
and let us recall that the action of welded braids is from left to  right  ($\beta_1 \beta_2 (x_i)$ denotes  $((x_i) \beta_1) \beta_2$). Therefore if we label $x_i$ (for $i=1 , \ldots, n$)  the arcs on  the top of the braid $\beta_w$, the arcs on the bottom 
will be labelled by $\beta_w^{-1} (x_i)$ (for $i=1 , \ldots, n$). Since we are considering the closure of $\beta_w$,  we identify  labels on the bottom with corresponding labels on the top and we deduce that a possible presentation of  $G_{K,w}(wK)$ is :
 $$
G_{K,w}(wK)= \langle x_1, x_2, \ldots , x_n~ ||~ x_i = \beta_w^{-1}(x_i),~~i=1, \ldots, n \rangle.
$$
and therefore $G_{K,w}(wK)$ is clearly isomorphic to $G_w(wK)$.
\end{proof}

\begin{figure}[h]
\centering
 \includegraphics[width=10cm,]{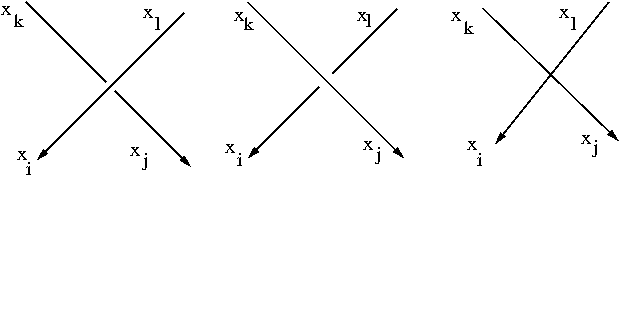}
\caption{Wirtinger-like labelling} \label{wirtweld}
\end{figure}

Let $\mathcal{D}$ be a diagram representing the immersion of a circle in the plane, where double points can
be presented  with two different labelings:
\begin{itemize}
\item with the usual overpasses/underpasses information;
\item as "singular" points.
\end{itemize}
Clearly $\mathcal{D}$ will represent a virtual
knot diagram if we allow virtual local moves or respectively a welded knot diagram if we allow welded local moves.

In the following we will write $G_v(\mathcal{D})$ and   $G_{K,v}(\mathcal{D})$  when we consider  $\mathcal{D}$
as a virtual knot diagram, while we will set   $G_w(\mathcal{D})$ and   $G_{K,w}(\mathcal{D})$ when we see $\mathcal{D}$
as a welded knot diagram.
The following Proposition is a consequence of the fact that Wirtinger labeling is preserved by $F1$ moves.

\begin{prop} \label{weldedvskauffman}
Let  $\mathcal{D}$ be a diagram as above. Then $G_{K,v}(\mathcal{D}) = G_{K,w}(\mathcal{D})$.
\end{prop}

Remark also that  if $\mathcal{D}_1$ and $\mathcal{D}_2$ are equivalent as welded diagrams
then $G_{K,v}(\mathcal{D}_1) = G_{K,v}(\mathcal{D}_2)$ even if  $\mathcal{D}_1$ and $\mathcal{D}_2$ are not equivalent as
virtual diagrams.

\medskip

\begin{proof}[Proof of Proposition \ref{isomprincipal}]
Let $\beta_v$ be a virtual braid with closure equivalent to $vK$.
 Let $F_{n+1} = \langle x_1, x_2, \ldots, x_n, y \rangle$  and $F_{n} = \langle x_1, x_2, \ldots, x_n \rangle$.
As in recalled in Proposition \ref{virtualvswelded}, the projection $p_n: F_{n+1} \to F_{n+1} / \langle \langle y  \rangle \rangle \simeq F_n$ induces a surjective map
$p^\#_n : \psi(VB_n) \to  WB_{n}$ and therefore   that $G_v(vK)/\langle \langle y  \rangle \rangle $ is isomorphic to
$G_w(wK) = \langle x_1, x_2, \ldots , x_n~ ||~ x_i = \beta_w(x_i),~~i=1, \ldots, n \rangle$, where  $\beta_w=p^\# (\psi(\beta_v)) $.

The first claim is therefore a straightforward consequence of Proposition \ref{weldedvskauffman}.

The second claim of the Proposition follows easily from the fact that
any generator of type $x_i$ is a conjugated of $x_{\pi(i)}$ where $\pi \in S_n$ is of order $n$. Therefore in the abelianization all generators of type $x_i$ have the same image.  
\end{proof}

From Propositions \ref{isomprincipal} and \ref{isomweld} therefore it follows that:

\begin{prop}
Let  $\mathcal{D}$ be a diagram as above. Then $G_{v}(\mathcal{D})/\langle\langle y \rangle\rangle = G_{w}(\mathcal{D})$.
\end{prop}

\section{Wada groups for virtual and welded links}

In \cite{W} Wada found several representations of $B_n$ in ${\rm Aut}(F_n)$ which, by the usual braid closure,
provide  group invariants of links. These representation are  of the following special form:
any generator (and therefore its inverse) of $B_n$ acts trivially on generators of $F_n$ except a pair of  generators:

\begin{eqnarray*} 
x_i \cdot \sigma_i=u^+(x_i, x_{i+1}) \,  ,\\
x_{i+1} \cdot \sigma_i=v^+(x_i, x_{i+1})  \,  , \\
x_j \cdot \sigma_i= x_j &\qquad j \not= i, i+1 \, .
\end{eqnarray*}

\begin{eqnarray*} 
x_i \cdot \sigma_i^{-1}=u^-(x_i, x_{i+1})  \,  ,\\
x_{i+1} \cdot \sigma_i^{-1}=v^-(x_i, x_{i+1})  \,  ,\\
x_j \cdot \sigma_i^{-1}= x_j &\qquad j \not= i, i+1 \, .
\end{eqnarray*}

where $u$ and $v$ are now words in the generators $a,b$,  with $\langle a, b \rangle \simeq F_2$.
In \cite{W} Wada found  four families of representations providing group invariants of links (they are types $4-7$ in Wada's paper):

\begin{itemize}
\item Type 1:  $u^+(x_i, x_{i+1})=  x_i^{h} x_{i+1} x_i^{-h}$ and  $v^+(x_i, x_{i+1})=   x_{i} $;
\item Type 2:  $u^+(x_i, x_{i+1})=  x_i x_{i+1}^{-1} x_i$ and  $v^+(x_i, x_{i+1})=   x_{i} $;
\item Type 3:  $u^+(x_i, x_{i+1})=  x_i x_{i+1} x_i$ and  $v^+(x_i, x_{i+1})=   x_{i}^{-1} $;
\item Type 4:  $u^+(x_i, x_{i+1})=  x_i^{2} x_{i+1}$ and  $v^+(x_i, x_{i+1})=   x_{i+1}^{-1}  x_{i}^{-1} x_{i+1} $.
\end{itemize}

As in the case of Artin representation we can ask if these representations extend to welded braids providing
group invariants for welded links. More precisely, let $\chi_k:  G_{WB_n} \to {\rm Aut}(F_n)$ (for $k=1, \ldots, 4$) be the set map from the set of generators $G_{WB_n}:=\{ \sigma_1, \ldots, \sigma_{n-1}, \alpha_1, \ldots, \alpha_{n-1} \}$ of $WB_n$
to  ${\rm Aut}(F_n)$ which associates to any generators $\sigma_i$ the Wada representation of type $k$ and  to any generators $\alpha_i$
the usual automorphism
\begin{eqnarray*} 
x_i \cdot \alpha_i=   x_{i+1}  &:=u^\circ(x_i, x_{i+1}) \\
x_{i+1} \cdot \alpha_i= x_i &:=v^\circ(x_i, x_{i+1}) \\
x_j \cdot \alpha_i= x_j &\qquad j \not= i, i+1 \, .
\end{eqnarray*}

\begin{prop}
The set map $\chi_k:  G_{WB_n} \to {\rm Aut}(F_n)$ (for $k=1, \ldots, 4$) induces a homomorphism
$\chi_k:  WB_n \to {\rm Aut}(F_n)$ if and only if $k=1,2$.
\end{prop}
\begin{proof}
For $k=1,2$ the proof  is a straightforward verification that relations of $WB_n$ hold in ${\rm Aut}(F_n)$: for $k=3,4$ it suffices to remark that relation of type $\alpha_i \sigma_{i+1} \sigma_i=  \sigma_{i+1} \sigma_i \alpha_{i+1}$ is not preserved.
 \end{proof}

When $k=1,2$ we can therefore define for any $\beta_w \in WB_n$ the group
$$\mathcal{W}_k(\beta_w)=\langle x_1, x_2, \ldots , x_n~ ||~ x_i = (\chi_k(\beta_w))(x_i),~~i=1, \ldots, n \rangle$$
that we will call Wada group of type $k$ for  $\beta_w \in WB_n$.

\begin{thm}
The Wada group $\mathcal{W}_k$ is a link invariant for $k=1,2$.
\end{thm}
\begin{proof}
One can repeat almost verbatim the arguments from \cite{W} for classical braids. Notice that $\mathcal{W}_k(\beta_w)$
is the group of co-invariants of $\beta_w$, i.e. the maximal quotient of $F_n$ on which  $\chi_k(\beta_w)$ acts trivially.
Since the group of co-invariants is invariant up to conjugation by an automorphism, we deduce that conjugated welded braids have isomorphic Wada groups.
To prove the statement it is therefore sufficient to verify that
$\mathcal{W}_k(\beta_w)=\langle x_1, x_2, \ldots , x_n~ ||~ x_i = (\chi_k(\beta_w))(x_i),~~i=1, \ldots, n \rangle$
and $\mathcal{W}^{stab}_k(\beta_w)=\langle x_1, x_2, \ldots ,  x_{n}x_{n+1} ~ ||~ x_i = (\chi_k(\beta_w \sigma_n))(x_i),~~i=1, \ldots, n+1 \rangle$ are isomorphic: this a straightforward computation similar to the case 3) in Theorem \ref{theorem4}, the key point being that $(\chi_k(\beta_w \sigma_n))(x_{n+1})=x_n$ for $k=1,2$ (see also Section 2 of \cite{W}).
 \end{proof}

In  \cite{ESSSW} Wada representations of type $2, 3, 4$ have been extended to group invariants for virtual links
using a Wirtinger like presentation of virtual link diagrams: contrarily to the classical case these  groups are not necessarily isomorphic.

Analogously it would be interesting to understand the geometrical meaning of
Wada groups of welded links. In this perspective, Proposition \ref{equiv}
shows that Wada representations of type $1$ and $2$ are not equivalent.

We will say that two representations $\omega_1: WB_n \to {\rm Aut}(F_n)$ and $\omega_2: WB_n \to {\rm Aut}(F_n)$ are equivalent, if there exist  automorphisms
$\phi \in {\rm Aut}(F_{n})$ and $\mu : WB_n \to WB_n$ such that
$$
\phi^{-1} \, \omega_1(\beta_w) \, \phi = \omega_2(\mu(\beta_w)) ,
$$
for any $\beta_w  \in G_{WB_n}$.

\begin{prop} \label{equiv}
Wada representations of type $1$ and $2$ are not equivalent.
\end{prop}

\begin{proof}
The proof is the same as in Proposition A.1 of \cite{CP}: the claim follows from considering
 the induced action on $H_1(F_n)$. Under Wada representations of type $1$ a welded braid has evidently  finite order as automorphism of
 $H_1(F_n)$ while we have that $\chi_2(\sigma_i^t)[x_i]=(t+1)[x_i] -t[x_2]$ for all $t \in \N$, where $[u]$ denote the equivalence class in
 $H_1(F_n)$ of an element $u \in F_n$.
\end{proof}


\begin{thebibliography}{MKS}
\bibitem{Ba} J.C. Baez, Link invariants of finite type and perturbation theory, Lett. Math. Phys. 26 (1992) 43 -- 51.


\bibitem{Bar-1}
V. G. Bardakov, Virtual and welded links and their invariants. Sib. Elektron. Mat. Izv. 2 (2005), 196--199 (electronic).


\bibitem{Bar}
V. G. Bardakov, The structure of the group of conjugating
automorphisms, Algebra i Logik, 42, No.5 (2003), 515-541.

 
\bibitem{BarP}
V. G. Bardakov, The structure of the group of conjugating
automorphisms and the linear representation of the braid groups
of some manifolds, Preprint (math.GR/0301247).


\bibitem{Bar-K04}
V. G. Bardakov,  Linear representations of the group of conjugating automorphisms and
the braid groups of some manifolds,
Siberian  Math. J., 46,  No.1 (2005), 17-31.


\bibitem{Bn}
D. Bar Natan and Z. Dancso,  Finite Type Invariants of W-Knotted Objects: From Alexander to Kashiwara and Vergne,
preprint available on {\it  http://www.math.toronto.edu/drorbn/papers/WKO/}

\bibitem{Bnweb}
 www.katlas.math.toronto.edu/drorbn/AcademicPensieve/Projects/WKO/nb/The$\_$Kishino$\_$Braid.pdf

\bibitem{BerP}
B. Berceanu and S. Papadima, Universal representations of braid and braid-permutation groups. J. Knot Theory Ramifications 18 (2009), no. 7,
999--1019.

\bibitem{BH}
T. Brendle and A. Hatcher,
Configuration spaces of rings and wickets, to appear in Commentarii Math. Helv.


\bibitem{ESSSW}
M. Elhamdadi, M. Saito, J. Scott Carter, D. Silver and  S. Williams,  Virtual knot invariants from group biquandles and their cocycles.
J. Knot Theory Ramifications 18 (2009), no. 7, 957--972.
\bibitem{CMG} J.H. Conway and C. McA. Gordon, A group to classify knots, Bull. London Math.
Soc. 7 (1975) 84-88


\bibitem{CP}
J. Crisp and  L. Paris,  Representations of the braid group by automorphisms of groups, invariants of links, and Garside groups. Pacific J. Math. 221 (2005), no. 1, 1--27.

\bibitem{FRR}
R. Fenn, R. Rim\'{a}nyi, C. Rourke, The braid--permutation group,
Topology, 36, No.1 (1997), 123-135.


\bibitem{G} B. Gemein,
Singular braids and Markov's theorem.
J. Knot Theory Ramifications 6 (1997), no. 4, 441--454.
 
\bibitem{GPV} M. Goussarov, M. Polyak, O. Viro, Finite-type invariants of classical and
virtual knots, Topology, 39, No.5 (2000), 1045-1068.

\bibitem{Gr} K. W. Gruenberg, Residual properties of infinite soluble groups, Proc. London Math. Soc. (3), 7 (1957),
29-62

\bibitem{He} J. Hempel,  Residual Finiteness for 3-manifolds, Combinatorial Group Theory and
Topology (Alta, Utah, 1984), Annals of Math. Studies 111 Princeton Univ. Press,
Princeton, 1987.

\bibitem{J} D. Joyce, A Classifying Invariant of Knots, the Knot Quandle, Journal of Pure and
Applied Algebra 23 (1982) 37-65.


\bibitem{K} S.~Kamada, Invariants of virtual braids and a remark on left
  stabilisations and virtual exchange moves, Kobe J. Math, 21, (2004), 33--49.

\bibitem{Kan}
T. Kanenobu, Forbidden moves unknot a virtual knot, J. Knot Theory Ramifications
10 (2001), 89--96.


\bibitem{Ka}
L. H. Kauffman, Virtual knot theory, Eur. J. Comb., 20, No.7
(1999), 663-690.

\bibitem{KaL}
L.H. Kauffman, S. Lambropoulou, Virtual braids and the L-move, J. Knot Theory and its Ramifications 15 (2006), 1--39.

\bibitem{LH} V.~Lundsgaard Hansen.  Braids and coverings. London Mathematical Society
Student Texts 18, Cambridge, 1989.

\bibitem{KMS}
W. Magnus, A. Karrass, D. Solitar, Combinatorial group theory,
Interscience Publishers, New York, 1996.

\bibitem{M}
V. O. Manturov,
On the recognition of virtual braids. (Russian) Zap. Nauchn. Sem. S.-Peterburg. Otdel. Mat. Inst. Steklov. (POMI) 299 (2003),
Geom. i Topol. 8, 267--286, 331--332; translation in J. Math. Sci. (N. Y.) 131 (2005), no. 1, 5409--5419

\bibitem{M-1}
V. O. Manturov, Knot theory. Chapman \& Hall/CRC, Boca Raton, FL, 2004.


\bibitem{Nel}
S. Nelson, Unknotting virtual knots with Gauss diagram forbidden moves, J. Knot Theory
Ramifications 10 (2001), 931--935.

\bibitem{S}
A. G. Savushkina,  On a group of conjugating automorphisms of a free group.
(Russian) Mat. Zametki 60 (1996), no. 1, 92--108, 159; translation in Math. Notes 60 (1996), no. 1-2, 68--80 (1997)

\bibitem{SW} D. Silver and  S. Williams, Virtual knot groups, Knots in Hellas '98, World Scientific 2000,  440-451.

\bibitem{Ver01}
V. V. Vershinin, On homology of virtual braids and Burau
representation, J. Knot Theory Ramifications, 10, No.5 (2001),
795-812.

\bibitem{W}
M. Wada, Group invariants of links.
Topology 31 (1992), no. 2, 399--406.

\bibitem{Wi}
B. Winter,
The classification of spun torus knots.
J. Knot Theory Ramifications 18 (2009), no. 9, 1287--1298.

\end{thebibliography}
\end{document}